\documentclass[12pt]{amsart}
\usepackage{amsmath, amssymb, euscript, hyperref, graphicx}

\newtheorem{thm}{Theorem}[section]
\newtheorem{lemma}[thm]{Lemma}
\newtheorem{prop}[thm]{Proposition}
\newtheorem{cor}[thm]{Corollary}

\newtheorem*{lemma:lemma_main}{Lemma \ref{lemma_main}}
\newtheorem*{thm:thm_main}{Theorem \ref{thm_main}}
\newtheorem*{cor:cor_main}{Corollary \ref{cor_main}}
\newtheorem*{thm:thm_also_1}{Theorem \ref{thm_also_1}}
\newtheorem*{thm:thm_also_2}{Theorem \ref{thm_also_2}}
\newtheorem*{cor:cor_also}{Corollary \ref{cor_also}}

\theoremstyle{definition}

\newtheorem{rem}[thm]{Remark}

\newcommand{\Homeo}{\mathop{\rm Homeo}}
\newcommand{\Diff}{\mathop{\rm Diff}}
\newcommand{\PL}{\mathop{\rm PL}}
\newcommand{\PLF}{\mathop{\rm PLF}}

\begin{document}

\author{Michael P. Cohen}
\address{Michael P. Cohen,
Department of Mathematics,
North Dakota State University,
PO Box 6050,
Fargo, ND, 58108-6050}
\email{michael.cohen@ndsu.edu}

\author{Robert R. Kallman}
\address{Robert R. Kallman,
Department of Mathematics,
University of North Texas,
1155 Union Circle 311430,
Denton, TX, 76203-5017}
\email{kallman@unt.edu}

\title{$\PL_+(I)$ is not a Polish group}

\thanks{\textbf{Acknowledgements.}  The first author would like to thank Azer Akhmedov for many motivating discussions and remarks.}

\maketitle

%\keywords{}
%\subjclass[2010]{}

\begin{abstract}  The group $\PL_+(I)$ of increasing piecewise linear self-homeomorphisms of the interval $I=[0,1]$ may not be assigned a topology in such a way that it becomes a Polish group.  The same statement holds for the groups $\Homeo_+^{Lip}(I)$ of bi-Lipschitz homeomorphisms of $I$, and $\Diff_+^{1+\epsilon}(I)$ of diffeomorphisms of $I$ whose derivatives are H\"older continuous with exponent $\epsilon$.
\end{abstract}

\section{Introduction}

Topologization is a fundamental tool in the study of transformation groups.  An abstract group is a purely algebraic object; a topological group, however, especially one equipped with a natural continuous action on some topological space, is at once an algebraic, a topological, and a dynamical object, and we may pursue an understanding of its properties via a powerful synthesis of these differing vantage points.  Thus researchers have equipped the lion's share of uncountable groups across all branches of mathematics with rich and interesting topologies to help facilitate their study.

We are particularly invested in the case where a group may be associated with a group topology which is \textit{Polish} (completely metrizable and separable).  The researcher of a Polish group comes armed with an arsenal of powerful tools, not least of which are the Baire category theorem and the vast body of descriptive-set-theoretic literature devoted to Polish group actions (see, among many others, \cite{becker_kechris_1996a}, \cite{hjorth_2000a}, \cite{gao_2009a}, \cite{pestov_2006a}).  A particularly interesting discovery of the last few decades is that many well-studied Polish groups actually have a \textit{unique} Polish topology, and therefore their topological structure is no contrivance but rather an intrinsic structure wholly determined by the group's algebraic definition.

Thus we are motivated by the foundational question:  Given an abstract group $G$ of cardinality continuum, may $G$ be assigned a topology which makes $G$ into a Polish group?  We consider this question in the context of the dynamics of $1$-dimensional manifolds.  Recall that $\PL_+(I)$ denotes the group of piecewise-linear increasing homeomorphisms of $I$ onto itself, which have only finitely many changes of slope.  There have been a large number of papers (see \cite{brin_squier_2001a}, \cite{bleak_2008a} for some important examples) devoted purely to the algebraic structure of $\PL_+(I)$.  These studies are largely motivated by a program to fully understand the properties of Thompson's group $F$, which embeds naturally as a finitely generated subgroup of $\PL_+(I)$, and shares many similarities with its host group.  The considerations of this paper stem from a desire to better understand the topological-algebraic structure of $\PL_+(I)$.  Namely, we prove the titular result.

\begin{thm} \label{thm_main}  There is no Polish group topology on $\PL_+(I)$.
\end{thm}

The tools used for establishing Theorem \ref{thm_main} may be of independent interest to researchers of groups of one-dimensional piecewise-linear maps.  Although $\PL_+(I)$ has no Polish group topology, we are able to regard $\PL_+(I)$ as a Polish topological \textit{space} (with discontinuous group multiplication) by identifying each map $f$ with its ordered $n$-tuple of break points $B(f)$ in $[(0,1)\times(0,1)]^n$.  This allows us to make several Baire category statements about the group which are central to the proof.  We expect that many other interesting Baire category statements about $\PL_+(I)$ can be formulated and proved.

The closely related corollary below follows with a brief proof.

\begin{cor} \label{cor_main}  There is no Polish group topology on any of the following groups:
\begin{enumerate}
		\item $\PL_+(\mathbb{R})$, the group of increasing piecewise-linear self-homeomorphisms of $\mathbb{R}$ whose set of break points is discrete in $\mathbb{R}$;
		\item $\PLF_+(\mathbb{R})$, the group of increasing piecewise-linear self-homeomorphisms of $\mathbb{R}$ whose set of break points is finite in $\mathbb{R}$; and
		\item $\PL_+(\mathbb{S}^1)$, the group of orientation-preserving piecewise-linear self-homeomorphisms of the circle $\mathbb{S}^1$.
\end{enumerate}
\end{cor}

Concretely, in the above we identify the circle $\mathbb{S}^1$ with the quotient space $\mathbb{R}/\mathbb{Z}$, and define $\PL_+(\mathbb{S}^1)$ to be the subgroup of $\PL_+(\mathbb{R})$ consisting of homeomorphisms $f$ which satisfy the equation $f(x+1)=f(x)+1$ for every $x\in\mathbb{R}$, i.e. which induce a homeomorphism of $\mathbb{R}/\mathbb{Z}$.

While we are at it, we apply the machinery of this paper to study the class of compatible topologies on several other well-studied groups of interval transformations.  Consider the following chain of subgroups:\\

\begin{center} $\Homeo_+(I)\geq\Diff_+(I)\geq\Diff_+^2(I)\geq...\geq\Diff_+^k(I)\geq...\geq\Diff_+^\infty(I)$
\end{center}
\vspace{.3cm}

These denote, listed in order, groups of homeomorphisms of the interval $I=[0,1]$: the full homeomorphism group, then the group of $C^1$-diffeomorphisms, followed by $C^2$-diffeomorphisms, $C^k$-diffeomorphisms $(k\geq 2)$, and lastly the infinitely differentiable diffeomorphisms.  Each of these groups is a Polish group when metrized by a suitable metric, and each topology is strictly finer than the preceding one.  The dynamics of these groups and their countable subgroups is a major focus of classical study.

We may think of the groups above as representing ``integer-valued'' degrees of regularity.  In recent decades, there has been blooming interest in groups of interval homeomorphisms which satisfy some form of ``in-between'' regularity: for an example, the group $\Homeo_+^{abs}(I)$ of homeomorphisms of $I$ which are absolutely continuous and have absolutely continuous inverse lies strictly between $\Homeo_+(I)$ and $\Diff_+(I)$ as a group.  Very interestingly, Solecki \cite{solecki_1999a} has shown that $\Homeo_+^{abs}(I)$ may be equipped with a Polish group topology strictly finer than that of $\Homeo_+(I)$ but strictly coarser than that of $\Diff_+(I)$.  Consider now a few other important ``in-between'' groups:\\

\begin{center} $\Homeo^{Lip}_+(I)=$ the group of increasing bi-Lipschitz self-homeomorphisms of $I$;\\
\vspace{.3cm}
$\Diff^{1+\epsilon}_+(I)=$ the group of increasing self-homeomorphisms of $I$ which are differentiable with differentiable inverse, and whose derivatives are H\"older continuous with exponent $0<\epsilon<1$.
\end{center}
\vspace{.3cm}

The first group is of a regularity class less than $C^1$, while the second lies strictly between $C^1$- and $C^2$-regularity.  The dynamical properties of these groups and their countable subgroups have received considerable attention in the literature---for some overview of facts and open problems we refer the reader to the well-known reference \cite{navas_2011a} or the recent survey article \cite{akhmedov_2013a}.  Our results indicate that these groups are fundamentally unlike those we have previously listed; they prohibit Polish topologization entirely.

\begin{thm} \label{thm_also_1}  There is no Polish group topology on $\Homeo^{Lip}_+(I)$.
\end{thm}

\begin{thm} \label{thm_also_2}  There is no Polish group topology on $\Diff^{1+\epsilon}(I)$.
\end{thm}

Again we get an easy but significant corollary.

\begin{cor} \label{cor_also}  There is no Polish group topology on either of the following groups:
\begin{enumerate}
		\item $\Homeo_+^{Lip}(\mathbb{R})$, the group of increasing bi-Lipschitz self-homeomorphisms of $\mathbb{R}$; and
		\item $\Homeo_+^{Lip}(\mathbb{S}^1)$, the group of orientation-preserving bi-Lipschitz self-homeomorphisms of $\mathbb{S}^1$.
\end{enumerate}
\end{cor}

We make some remarks about extending our conclusions to the group $\Diff_+^{1+\epsilon}(\mathbb{S}^1)$; see Remark \ref{rmk_extend}.

We wish to make a few comments about the techniques used in this paper.  The purely abstract question of when a given group $G$ admits some Polish topology seems unapproachable in many cases.  There are very few known examples of groups of cardinality continuum whose algebraic structure precludes the existence of a Polish group topology.  Perhaps the earliest class of examples is given by the so-called \textit{normed groups} of Dudley's classical paper \cite{dudley_1961a}.  Dudley's results imply that there is no Polish topology compatible with the group structure of, for instance, the free group on continuum-many generators or the free abelian group on continuum-many generators.  Hjorth gives another interesting example of a group with no compatible Polish topology in \cite{hjorth_2000a} Ch. 8 Ex. 8.6, while Rosendal displays a variety of such objects in \cite{rosendal_2005a}, including the homeomorphism groups $\Homeo\mathbb{Q}$ and $\Homeo\omega^\omega$ of the rationals and the Baire space, respectively.

An alternative formulation of our motivating question is often more manageable in concrete cases: suppose $G$ is an abstract group and $\Sigma$ is a distinguished $\sigma$-algebra of subsets of $G$.  May $G$ be assigned a Polish group topology in such a way that $\Sigma$ becomes exactly the class of Borel subsets of $G$?  If $G$ admits a Polish group topology generating $\Sigma$ as its Borel sets, then we say $G$ is a \textit{Polishable} group.  Most commonly, this question arises in the context where $G$ is a Borel subgroup of some given Polish group $H$, and $\Sigma$ is the class of Borel subsets of $H$ intersected with $G$.  In this case the additional requirement that the Polish group topology must generate $\Sigma$ as its Borel sets is natural as it is no stronger than requiring that the natural injection of $H$ into $G$ be continuous; in addition, if such a Polish topology exists, then it is unique.  Questions about Polishability are often easier to answer, as knowing which sets are in $\Sigma$ gives additional control over the problem.

%As an example, suppose $H=\Homeo_+(I)$ equipped with the compact-open topology, and let $G$ be an abstract subgroup of $H$.  If $G$ admits some Polish group topology under which the point evaluation maps $g\mapsto g(x)$, $G\rightarrow I$ ($x\in I)$ are continuous, then we may deduce from classical theorems of Arens (\cite{arens_1946a}) and Lusin-Souslin (\cite{kechris_1995a} Thm. 15.1) that $G$ is in fact Borel in $H$, and its Borel structure as a Polish group coincides exactly with its Borel structure as a Borel subgroup of $H$.  (See the proof of Lemma \ref{lemma_main} for the details.)  So asking for the group $G$ to be Polishable (respecting its Borel sets inherited from $H$) is no stronger than asking for a Polish topology on $G$ under which the point evaluation maps are continuous: a modest goal given that $G$ is a homeomorphism group in the first place!

The key lemma of this paper is that the abstract problem and the Polishability problem are in fact one and the same for sufficiently rich subgroups of $\Homeo_+(I)$.  This allows us to give new examples of groups which preclude Polish topologization entirely.

\begin{lemma} \label{lemma_main}  Let $G$ be a subgroup of $\Homeo_+(I)$ with the following property: whenever $U\subseteq I$ is a nonempty open set, there is a non-identity map $g_U\in G$ which fixes $I\backslash U$ pointwise.  Let $\Sigma$ be the class of Borel subsets of $G$ inherited as a subgroup of $\Homeo_+(I)$.  Then $G$ admits some Polish group topology if and only if $(G,\Sigma)$ is Polishable.
\end{lemma}

If a group satisfies the condition given in the lemma, then it contains a rich system of ``small perturbations.''  All of the homeomorphism groups we have discussed so far satisfy this condition.

Lemma \ref{lemma_main} follows directly from previous results of Arens (\cite{arens_1946a}), Kallman (\cite{kallman_1986a}), and Hayes (\cite{hayes_1997a}) which we state and prove in Section 2.  In Section 3 we topologize $\PL_+(I)$ as a Polish space and make some easy Baire category observations.  The proofs of Theorem \ref{thm_main} and Corollary \ref{cor_main} lie in Section 4, while Theorems \ref{thm_also_1} and \ref{thm_also_2} and Corollary \ref{cor_also} are proven in Section 5.

\section{Lemmas and Observations}

In this section we obtain Lemma \ref{lemma_main}, as well as recall some well-known Baire category paraphernalia that will be relevant in later sections.  Lemma \ref{lemma_kallman} is due to Kallman (\cite{kallman_1986a}) and Lemma \ref{lemma_hayes} may be extracted from the 1997 UNT Ph.D. dissertation of Diana Hayes (\cite{hayes_1997a}).  The original sources state the lemmas in considerably greater generality.  We tailor the lemmas to our immediate goals, and for completeness of exposition, we include brief proofs.

Recall that if $G$ is a group and $\tau$ is a topology on $G$, then the pair $(G,\tau)$ is called a topological group if both the multiplication mapping $(g,h)\mapsto gh$, $G\times G\rightarrow G$ and the inversion mapping $g\mapsto g^{-1}$, $G\rightarrow G$ are $\tau$-continuous.  In this case we call $\tau$ a group topology on $G$.  In order to get a handle on what the possible group topologies look like for the groups we are considering, we need the following purely algebraic lemma.

\begin{lemma} \label{lemma_intermediate}  Let $G$ be a subgroup of $\Homeo_+(I)$ with the following property: whenever $U\subseteq I$ is a nonempty open set, there is a non-identity map $g_U\in G$ which fixes $I\backslash U$ pointwise.  Let $W\subseteq G$ be open.  If $f\in G$ commutes with $g_{W'}$ for every open nonempty $W'\subseteq W$, then $f$ fixes $W$ pointwise.
\end{lemma}

\begin{proof}  Suppose contrapositively that there exists $x\in W$ with $f(x)=y\neq x$.  Let $\delta>0$ be so small that $2\delta<|y-x|$, and set $W'=W\cap(x-\delta,x+\delta)\cap f^{-1}(y-\delta,y+\delta)$.  So $W'\subseteq W$ and $f(W')\cap W'=\emptyset$.  Since $g_{W'}$ is not the identity map, find $z\in W'$ for which $g_{W'}(z)\neq z$.  Then $f(z)\notin W'$ and hence $g_{W'}f(z)=f(z)$.  But $g_{W'}(z)\neq z$, so $fg_{W'}(z)\neq f(z)$ since $f$ is a bijection.  This shows $f$ and $g_{W'}$ do not commute, which proves the lemma.
\end{proof}

\begin{lemma}[Kallman 1986] \label{lemma_kallman}  Let $G$ be a subgroup of $\Homeo_+(I)$ with the following property: whenever $U\subseteq I$ is a nonempty open set, there is a non-identity map $g_U\in G$ which fixes $I\backslash U$ pointwise.  For each pair of nonempty open sets $U,V\subseteq I$ let\\

\begin{center} $C(U,V)=\{f\in G:f(\overline{U})\subseteq \overline{V}\}$.
\end{center}
\vspace{.3cm}

If $\tau$ is a Hausdorff group topology on $G$, then each $C(U,V)$ is closed in $(G,\tau)$.
\end{lemma}

\begin{proof}  If $\overline{V}=I$ then $C(U,V)=G$, in which case there is nothing to prove, so assume $\overline{V}\neq I$.  Let $U'$ and $W'$ be arbitrary nonempty open subsets of $U$ and $I\backslash \overline{V}$ respectively. Let $g_{U'}$ and $g_{W'}$ be the maps associated to $U'$ and $W'$, respectively, as given by the hypothesis.  Since $\tau$ is a group topology, the mapping $\phi(f)=[fg_{U'}f^{-1},g_{W'}]$, $\phi:G\rightarrow G$ is continuous (where $[\cdot,\cdot]$ denotes the ordinary commutator).  Since $(G,\tau)$ is Hausdorff, this implies that the preimage of the identity\\

\begin{center} $F(U',W')=\phi^{-1}(e)$
\end{center}
\vspace{.3cm}

\noindent is closed in $(G,\tau)$.  Thus to finish the proof it suffices to observe that\\

\begin{center} $C(U,V)=\displaystyle\bigcap\{F(U',W'):U'\subseteq U$ open, $W'\subseteq I\backslash \overline{V}$ open$\}$,
\end{center}
\vspace{.3cm}

\noindent whence $C(U,V)$ is an intersection of closed sets and hence closed.

Suppose $f$ lies in the right-hand side above, i.e. suppose for all $U'\subseteq U$ open and $W'\subseteq I\backslash \overline{V}$ open, $fg_{U'}f^{-1}$ commutes with $g_{W'}$.  Lemma \ref{lemma_intermediate} implies that for each open $U'\subseteq U$, $fg_{U'}f^{-1}$ fixes $I\backslash \overline{V}$ pointwise.  This in turn implies that $g_{U'}$ fixes $f^{-1}(I\backslash \overline{V})$ pointwise, for each open $U'\subseteq U$.  Since $U'$ ranges over all open subsets of $U$, we have $U\cap f^{-1}(I\backslash \overline{V})=\emptyset$.  In other words $f(U)\subseteq \overline{V}$, hence $f(\overline{U})\subseteq \overline{V}$, and $f$ lies in the left-hand side above.  These arguments also work in reverse, so set equality holds and the proof is complete.
\end{proof}

\begin{lemma}[Hayes 1997] \label{lemma_hayes}  Let $G$ be a subgroup of $\Homeo_+(I)$ with the following property: whenever $U\subseteq I$ is a nonempty open set, there is a non-identity map $g_U\in G$ which fixes $I\backslash U$ pointwise.  If $\tau$ is a Hausdorff group topology on $G$, then $\tau$ contains the compact-open topology.
\end{lemma}

\begin{proof}  Arens (\cite{arens_1946a}) showed that if $G$ is a subgroup of $\Homeo_+(I)$ and $\tau$ is a group topology on $G$ for which the point evaluation maps $g\mapsto g(x)$, $(G,\tau)\rightarrow I$ are all continuous for $x\in I$, then $\tau$ must contain the compact-open topology.  So it suffices for us to check that point-evaluation maps are continuous.

To that end, let $x\in I$ be arbitrary.  If $x=0$ or $x=1$, then every element of $G$ fixes $x$, whence continuity is automatic, so assume that $x\in (0,1)$.

We start by showing $f\mapsto f(x)$ is continuous at the identity $e\in G$.  Let $U\subseteq I$ be an open neighborhood of $x$.  Assume $U$ contains neither $0$ nor $1$.  Let $\epsilon>0$ be so small that $[x-\epsilon,x+\epsilon]\subseteq U$, and let $\delta=\frac{\epsilon}{2}$.  Set $V_1=(x-\epsilon,x-\delta)$ and $V_2=(x+\delta,x+\epsilon)$.  By Lemma \ref{lemma_kallman}, the set\\

\begin{center} $F=C(V_1,I\backslash\overline{V_1})\cup C(V_2,I\backslash\overline{V_2})$
\end{center}
\vspace{.3cm}

\noindent is closed in $(G,\tau)$.  Hence the complementary set $W=G\backslash F$ is a $\tau$-open neighborhood of identity.  For every map $f\in W$, we have $\overline{V_1}\cap f(\overline{V_1})\neq\emptyset$ and $\overline{V_2}\cap f(\overline{V_2})\neq\emptyset$.  Since each such $f$ preserves the linear order on $I$, it follows that $f(x)\in[x-\epsilon,x+\epsilon]\subseteq U$ for all $f\in W$.  This shows $f\mapsto f(x)$ is continuous at $e\in G$ for each $x\in X$.

For continuity everywhere, let $x\in I$ and $f_0\in G$ be arbitrary, and let $U\subseteq I$ be an open set containing $f_0(x)$.  Since $f_0^{-1}$ is a homeomorphism, $f_0^{-1}(U)$ is a neighborhood of $x$, and thus by continuity at identity there is $\tau$-open $W\subseteq G$ with $W\cdot x\subseteq f_0^{-1}(U)$.  Then $f_0W$ is a $\tau$-open neighborhood of $f_0$ in $G$, and $f_0W\cdot x\subseteq f_0(f_0^{-1}(U))=U$, showing continuity of $f\mapsto f(x)$ at $f_0$.
\end{proof}

\begin{lemma:lemma_main}  Let $G$ be a subgroup of $\Homeo_+(I)$ with the following property: whenever $U\subseteq I$ is a nonempty open set, there is a non-identity map $g_U\in G$ which fixes $I\backslash U$ pointwise.  Let $\Sigma$ be the class of Borel subsets of $G$ inherited as a subgroup of $\Homeo_+(I)$.  Then $G$ admits some Polish group topology if and only if $(G,\Sigma)$ is Polishable.
\end{lemma:lemma_main}

\begin{proof}  If $(G,\Sigma)$ is Polishable then obviously $G$ has a compatible Polish topology.  Conversely, suppose $(G,\tau)$ is a Polish group.  By Lemma \ref{lemma_hayes}, $\tau$ contains the compact-open topology and hence the identity injection $\varphi:G\rightarrow\Homeo_+(I)$ is continuous (where $\Homeo_+(I)$ carries the compact-open topology).  By the Lusin-Souslin theorem (\cite{kechris_1995a} Thm. 15.1), $\varphi$ is a Borel isomorphism of $G$ with $\varphi(G)=G$; thus $\tau$ is a topology on $G\cong\varphi(G)$ which realizes $\Sigma$ as its family of Borel sets.
\end{proof}

The task of showing $\PL_+(I)$ has no Polish topology is now reduced to the much simpler task of showing it is not Polishable as a subgroup of $\Homeo_+(I)$.  There is a somewhat standard argument for showing non-Polishability.  Rather than rephrasing the argument multiple times in this paper, we extract it into Lemma \ref{lemma_standard} below.  It relies on the beautiful theorem of Pettis \cite{pettis_1950a}, which we recall now.

\begin{thm}[Pettis] \label{thm_pettis}  Let $G$ be a Polish group and let $A\subseteq G$ be a set with the Baire property which is non-meager.  Then $AA^{-1}$ and $A^{-1}A$ contain a neighborhood of identity in $G$.
\end{thm}

\begin{lemma} \label{lemma_standard}  Let $(G,\Sigma)$ be a standard Borel group.  Suppose there exists a sequence of subsets $(B_n)_{n\in\omega}$ with the following properties:
\begin{enumerate}
		\item $G=\displaystyle\bigcup_{n\in\omega}B_n$;\\
		\item for each $n$, $B_n$ is symmetric (i.e. $B_n^{-1}=B_n$) and $B_n\in\Sigma$;\\
		\item for each $n$, there exists $m_n\in\omega$ such that $B_n^2\subseteq B_{m_n}$; and\\
		\item for each $n$, whenever $(g_k)_{k\in\omega}$ is a sequence of elements in $G$, there exists $f\in G$ such that $f\notin\displaystyle\bigcup_{k\in\omega}g_kB_n$ (or $f\notin\displaystyle\bigcup_{k\in\omega}B_ng_k$).
\end{enumerate}
Then $(G,\Sigma)$ is not Polishable.
\end{lemma}

\begin{proof}  Suppose toward a contradiction that $\tau$ is a Polish topology on $G$ generating $\Sigma$ as its Borel sets.  Then each $B_n$ is Borel in $(G,\tau)$ by condition (2) and hence has the Baire property.  $(G,\tau)$ is non-meager in itself, and thus by condition (1), there must be some particular $n$ for which $B_n$ is non-meager.  By Pettis's Theorem \ref{thm_pettis}, $B_n^2=B_n^{-1}B_n$ contains a $\tau$-open neighborhood of identity, whence the set $B_{m_n}$ of condition (3) contains a $\tau$-open set also.  Since $(G,\tau)$ is second-countable, it follows that there exists a sequence $(g_k)_{k\in\omega}$ in $G$ for which $\{g_kB_{m_n}:k\in\omega\}$ (or $\{B_{m_n}g_k:k\in\omega\}$) covers $G$, contradicting (4).
\end{proof}

\section{$\PL_+(I)$ as a Polish topological space}

Our main tool for the proof of Theorem \ref{thm_main} is Lemma \ref{lemma_category}, a Baire category statement about multiplication in $PL_+(I)$.  For such a statement to make sense, we must observe that $PL_+(I)$ can be viewed as a Polish topological space in a natural way.

Let us introduce some notation.  We employ only the most rudimentary tools developed by Brin and Squier in their thorough study of $\PL_+(I)$ as an abstract group in \cite{brin_squier_2001a}.  For a map $f\in PL_+(I)$ and a point $x$ in $(0,1)$, let $f'_+(x)$ and $f'_-(x)$ denote the right- and left-hand derivatives of $f$ at $x$, respectively.  Define $f^*(x)=\frac{f'_+(x)}{f'_-(x)}$, the \textit{slope ratio} of $f$ at $x$, for all $x\in I$.  Note that since both the functions $f'_+$ and $f'_-$ satisfy the usual chain rule, their ratio $f^*$ also satisfies its own version of the chain rule, i.e. $(fg)^*(x)=f^*(g(x))g^*(x)$ for all maps $f$ and $g$ and all points $x\in(0,1)$.  By the definition of $PL_+(I)$, for each map $f\in PL_+(I)$ there are only finitely many points $x\in(0,1)$ for which $f^*(x)\neq1$.  We name these points the \textit{break points} of $f$, and denote by $B(f)$ the set of all such points.

For each $n\in\omega$ define\\

\begin{center} $A_n=\{f\in PL_+(I):\#B(f)=n\}$,
\end{center}
\vspace{.3cm}

Note that according to our definition $A_0=\{e\}$, but $A_n$ is an uncountably infinite set for each $n>0$, and $PL_+(I)=\displaystyle\bigcup_{n\in\omega}A_n$.

We intend to topologize $PL_+(I)$ by identifying each map $f$ in $A_n$ with its set of break points in the Polish product space $[(0,1)\times(0,1)]^n$.  To simplify our notation, we set $Z=(0,1)\times(0,1)$, and for each pair $z\in Z$ we permanently associate the symbols $x$ and $y$ to denote the first and second coordinates of $z$ respectively, so $z=(x,y)$.  Thus if $(z_1,z_2,...,z_n)$ is an $n$-tuple in $Z^n$, we assume the notation $z_i=(x_i,y_i)$ for each $1\leq i\leq n$.

With this assumption in place, let us say that an $n$-tuple $(z_1,z_2,...,z_n)\in Z^n$ \textit{defines a piecewise linear map} if both of the following two properties hold:\\

\begin{enumerate}
		\item $x_i<x_{i+1}$ and $y_i<y_{i+1}$ for every $1\leq i\leq n$; and\\
		\item no three consecutive points $z_{i-1},z_i,z_{i+1}$ are collinear in $Z$, where $1\leq i\leq n$, and we define $z_0=(0,0)$ and $z_{n+1}=(1,1)$.
\end{enumerate}
\vspace{.3cm}

For each $n\geq1$ let\\

\begin{center} $X_n=\{(z_1,z_2,...,z_n)\in Z^n:(z_1,z_2,...,z_n)$ defines a piecewise linear map$\}$.
\end{center}
\vspace{.3cm}

It is rather clear that if an $n$-tuple $(z_1,...,z_n)$ satisfies both (1) and (2), then so will any small perturbation of this tuple.  In other words, $X_n$ is an open subspace of $Z^n$, hence Polish.  

We define a map $\Phi_n:A_n\rightarrow X_n$ by letting $\Phi(f)=(z_1,z_2,...,z_n)$ be the $n$-tuple of coordinates $z_i=(x_i,f(x_i))$, where $x_1,...,x_n$ are the break points of $f$ ordered from left to right.  We also set $X_0=A_0=\{e\}$ (where the singleton space is equipped with its only possible topology, trivially a Polish topology) and let $\Phi_0$ be the identity on $A_0$.  Then for each $n\in\omega$, $\Phi_n$ is a bijection of $A_n$ onto $X_n$.  If $X$ denotes the disjoint union of the sets $X_n$ and $\Phi:PL_+(I)\rightarrow X$ is the union of the functions $\Phi_n$, then $\Phi$ is a bijection of $PL_+(I)$ with the Polish topological space $X$.  Thus, we associate to $PL_+(I)$ a Polish topology $\tau_0$ by declaring\\

\begin{center} $U$ is open in $(PL_+(I),\tau_0)~\longleftrightarrow~\Phi(U)$ is open in $X$.
\end{center}
\vspace{.3cm}

Note that each $A_n$ is clopen in this topology on $\PL_+(I)$.

Let us prove and apply an easy Baire category lemma.

\begin{lemma} \label{lemma_easy}  Let $M$ be a Baire space and $A\subseteq M$ a nowhere dense set.  Let $n\geq 1$ be an integer.  Then the set $E=\{(m_1,...,m_n)\in M^n:\exists i~m_i\in A\}$ is nowhere dense in $M^n$.
\end{lemma}

\begin{proof}  For $1\leq i\leq n$, set $E_i=\{(m_1,...,m_n)\in M^n:m_i\in A\}$, so $E$ is the finite union $E=\displaystyle\bigcup_{1\leq i\leq n}E_i$.  If $\pi_i$ is projection onto the $i$-th coordinate in $M^n$, then each $E_i=\pi_i^{-1}(A)$ is a nowhere dense set, whence $E$ is nowhere dense.
\end{proof}

\begin{cor} \label{corollary_useful}  Let $y_0\in(0,1)$ be fixed.  Then the set $\{g\in A_n:\exists x_0\in B(g)~g(x_0)=y_0\}$ is nowhere dense in $(A_n,\tau_0)$ for every $n\geq 1$.
\end{cor}

\begin{proof}  Note that the image under $\Phi$ of the set defined in the lemma is contained in $E\subseteq Z^n$, where\\

\begin{center} $E=\{(z_1,...,z_n)\in Z^n:\exists i~y_i=y_0\}$.
\end{center}
\vspace{.3cm}

If we set $A=\{z=(x,y)\in Z:y=y_0\}$, then $A$ is just the horizontal cross section of the box $Z=(0,1)\times(0,1)$ at $y_0$, clearly a nowhere dense set in $Z$.  Then $E$ is nowhere dense by Lemma \ref{lemma_easy}, and the corollary follows.
\end{proof}

Note that if $f\in PL_+(I)$ has exactly $n$ break points and $g\in PL_+(I)$ has exactly $m$ break points, then the number of break points of the product $fg$ is bounded above by $n+m$.  To see this, note that $B(fg)=\{x\in(0,1):(fg)^*(x)\neq 1\}$.  If $f^*(x)\neq 1$, then either $f^*(g(x))\neq1$ or $g^*(x)\neq 1$, since $(fg)^*(x)=f^*(g(x))g^*(x)$; it then follows that either $g(x)$ is a break point of $f$ or $x$ is a break point of $g$, i.e. $B(fg)\subseteq g^{-1}(B(f))\cup B(g)$ and $\#B(fg)\leq\#g^{-1}(B(f))+\#B(g)=\#B(f)+\#B(g)=n+m$.

The next lemma asserts that under the topology $\tau_0$, for a fixed $f\in A_n$, ``almost every'' product of the form $fg$ where $g_n\in A_m$ will have the maximal number of break points $n+m$, and consequently (perhaps counterintuitively), if $k>m$ then ``almost no'' maps $h\in A_k$ are of the form $h=fg$ for any map $g\in A_m$.

\begin{lemma} \label{lemma_category}  Let $n,m\in\omega$ be arbitrary, and equip $PL_+(I)$ with the topology $\tau_0$.  For every $f\in A_n$,
\begin{enumerate}
		\item $\{g\in A_m:fg\in A_{n+m}\}$ is comeager in $A_m$, and
		\item if $k>m$, $fA_m\cap A_k$ is meager in $A_k$.
\end{enumerate}
\end{lemma}

\begin{proof}  (\textit{Proof of (1).})  Let $E=\{g\in A_m:fg\in A_{n+m}\}$.  Let $B(f)=\{a_1,...,a_n\}$.  We claim that if $g$ is a map lying in the complement $A_m\backslash E$, then there must be some $x_0\in B(g)$ and some $1\leq i\leq n$ for which $g(x_0)=a_i$.  Suppose for a contradiction that this is not the case.  Enumerate $B(g)=\{b_1,...,b_m\}$, so $g(b_j)\neq a_i$ for each $1\leq i\leq n$, $1\leq j\leq m$.  For every $i$ we have $(fg)^*(g^{-1}(a_i))=f^*(a_i)g^*(g^{-1}(a_i))=f^*(a_i)\cdot 1\neq 1$, and for every $j$ we have $(fg)^*(b_j)=f^*(g(b_j))g^*(b_j)=1\cdot g^*(b_j)\neq 1$.  Therefore each of $g^{-1}(a_1),...,g^{-1}(a_n),b_1,...,b_m$ are distinct members of $B(fg)$ and $fg$ has $n+m$ break points, contradicting our choice of $g$ from the complement of $E$.  This proves the claim.

It follows then that\\

\begin{center} $A_m\backslash E\subseteq\displaystyle\bigcup_{1\leq i\leq n}\{g\in A_m:\exists x_0\in B(g)~g(x_0)=a_i\}$.
\end{center}
\vspace{.3cm}

The right-hand side above is a finite union of nowhere dense sets by Corollary \ref{corollary_useful}, and hence nowhere dense.  So $E$ is open dense in $A_m$.\\

(\textit{Proof of (2).})  Let $k>m$.  Note that $\#B(f^{-1})=\#B(f)=n$.  Then part (a) implies that for a comeager set of $h\in A_k$, $f^{-1}h$ has $k+n$ break points.  Any $h$ in this comeager set must not lie in $fA_m$, for if it did, and $h=fg$ for some $g\in A_m$, then we would have $k+n=\#B(f^{-1}h)=\#B(f^{-1}fg)=\#B(g)=m$, a contradiction.  Therefore $fA_m\cap A_k$ is meager in $A_k$.
\end{proof}

\section{Proof of Theorem \ref{thm_main}}

We have almost everything we need to prove Theorem \ref{thm_main}.  The last ingredient is to verify that each set $A_n$ is in fact Borel in $\Homeo_+(I)$.

\begin{prop}  For each $n\in\omega$, $A_n$ is $\Delta^0_3$ in $\Homeo_+(I)$.
\label{prop_structure}
\end{prop}

\begin{proof}  First consider the set $A_1$.  Let $\Phi_1:A_1\rightarrow X_1$ denote the bijection between $A_1$ and $X_1=\{(x,y)\in(0,1)\times(0,1):x\neq y\}$ defined in the previous section.  Denote $f=\Phi_1^{-1}$, so $f$ is a bijection of $X_1$ onto $A_1$.  It is clear from the definition that $f$, viewed as a mapping from the topological space $X_1$ into the topological group $\Homeo_+(I)$ is continuous.  Then since $X_1$ is $K_\sigma$, so too is $A_1=f(X_1)$ in $\Homeo_+(I)$.

%First let us observe that $A_1$ is $K_\sigma$.  Suppose $p,q,r,s$ are rational numbers satisfying $0<p<q<1$ and $0<r<s<1$, and such that either $q<r$ or $s<p$.  Let

%\begin{center}  $K_{p,q,r,s}=\{f\in A_1:\exists x\in I~B(f)=\{x\}$ and $(x,f(x))\in[p,q]\times[r,s]\}$.
%\end{center}
%\vspace{.3cm}

%Suppose $(f_k)_{k\in\omega}$ is a uniformly convergent sequence of maps in $K_n$.  Each $f_k$ has exactly one break point $x_k\in(0,1)$.  Set $y_k=f(x_k)$.  Since $[p,q]\times[r,s]$ is compact, the sequence $(x_k,y_k)$ has a cluster point $(x,y)$ in $[p,q]\times[r,s]$; assume by passing to a subsequence that $(x_k,y_k)\rightarrow (x,y)$.  Now either $x_k\leq x$ infinitely often or else $x_k\geq x$ infinitely often; for this argument assume $x_k\geq x$ infinitely often (the other case is similar).  By passing to a subsequence once more, assume $x_k\geq x$ for all $k$.  Then for every $t\in[0,x]$ we have $f_k(t)=\frac{y_k}{x_k}t$ while $f(t)=\frac{y}{x}t$.  It is easy to see from the formulas that $f_k\rightarrow f$ uniformly on $[0,x]$.  Likewise for any $t\in(x,1]$, for all sufficiently large values of $k$, we have $f_k(t)=\frac{1-y_k}{1-x_k}(t-x_k)+y_k$ and $f(t)=\frac{1-y}{1-x}(t-x)+y$, so $f_k\rightarrow f$ uniformly on $(x,1]$.  Thus $f_k\rightarrow f$ uniformly and $K_{p,q,r,s}$ is closed in $\Homeo_+(I)$. 

%$K_{p,q,r,s}$ is also bounded and equicontinuous, so is compact by the Arzela-Ascoli theorem.  Since $A_1$ is the union of the sets $K_{p,q,r,s}$ over all permissible choices of rationals $p,q,r,s$, $A_1$ is $K_\sigma$.

Now for each $n\geq1$ let $B_n=\displaystyle\bigcup_{i=0}^n A_i=\{f\in\PL_+(I):\#B(f)\leq n\}$.  $B_1=A_1\cup\{e\}$ is $K_\sigma$ by the previous paragraph.  By Lemma 1.10 of Brin and Squier \cite{brin_squier_2001a}, each map $f\in A_n$ may be written as the product $f=g_1g_2...g_n$ of elements $g_1,...,g_n\in B_1$.  In other words $B_n=(B_1)^n$, and since $B_1$ is $K_\sigma$, so too is $B_n$.  Therefore $A_n=B_n\backslash B_{n-1}$ is $\Delta^0_3$ for all $n\geq 2$.
\end{proof}

As a remark, since $\PL_+(I)$ is algebraically generated by the $K_\sigma$ set $B_1$, it too is $K_\sigma$ in $\Homeo_+(I)$.

\begin{thm:thm_main}  There is no Polish group topology on $\PL_+(I)$.
\end{thm:thm_main}

\begin{proof}[Proof of Theorem \ref{thm_main}]  By Lemma \ref{lemma_main} it suffices to show that $(\PL_+(I),\Sigma)$ is not Polishable, where $\Sigma=\mathcal{B}(\PL_+(I))$ denotes the class of Borel subsets of $\PL_+(I)$ viewed as a topological subgroup of $\Homeo_+(I)$.  Let $(A_n)_{n\in\omega}$ be the subsets of $\PL_+(I)$ defined in the previous section.  For each $n$, let $B_n=\displaystyle\bigcup_{i=0}^n A_i=\{f\in\PL_+(I):\#B(f)\leq n\}$.  We claim the sequence $(B_n)$ satisfies all the hypotheses of Lemma \ref{lemma_standard}.  Clearly $\PL_+(I)=\displaystyle\bigcup_{n\in\omega}B_n$, so (1) holds.  Each $B_n$ is symmetric and Borel in $\Homeo_+(I)$ by Proposition \ref{prop_structure}, i.e. $B_n$ is in $\Sigma$, showing that (2) holds.  (3) holds because $B_n^2\subseteq B_{2n}$ for each $n$.

To see that (4) holds: if $n$ is fixed and $(g_k)_{k\in\omega}$ is an arbitrary sequence of elements of $G$, then each left translate $g_kB_n$ is meager in $(A_{n+1},\tau_0)$ (where $\tau_0$ is the topology defined in Section 3) by the second part of Lemma \ref{lemma_category}.  It follows that $A_{n+1}$ is not covered by any union of the form $\displaystyle\bigcup_{k\in\omega}g_kB_n$, hence neither is $\PL_+(I)$.  We conclude by Lemma \ref{lemma_standard} that $\PL_+(I)$ is not Polishable, and therefore has no Polish topology whatsoever by Lemma \ref{lemma_main}.

%Suppose that $(PL_+(I),\tau)$ is a Polish group whose Borel structure coincides with that induced by the subspace topology inherited from $\Homeo_+(I)$.  By Proposition \ref{prop_structure}, each set $A_n\subseteq PL_+(I)$ is a $K_\sigma$ subset of $PL_+(I)$, and hence Borel in $(PL_+(I),\tau)$.  Since $PL_+(I)=\displaystyle\bigcup_{n\in\omega}A_n$ and $(PL_+(I),\tau)$ is Polish, at least one of these sets $A_n$ must be nonmeager in $(PL_+(I),\tau)$ by the Baire category theorem.  Fix a particular $n$ so that $A_n$ is nonmeager.

%By Pettis's lemma, $A_n^{-1}A_n$ contains an open neighborhood of identity in $(PL_+(I),\tau)$.  But for every $f,g\in A_n$, $\#B(f)=\#B(g^{-1})=n$, and therefore $\#B(g^{-1}f)\leq\#B(g^{-1})+\#B(f)=2n$.  It follows that $A_n^{-1}A_n\subseteq A$, where we define $A=\displaystyle\bigcup_{1\leq i\leq 2n}A_i$.  Thus $A$ contains an open neighboorhood of identity.

%Since $(PL_+(I),\tau)$ is Polish, the group may therefore be covered by countably many left translates of $A$, i.e. $PL_+(I)\subseteq\displaystyle\bigcup_{k\in\omega}f_kA$ for some sequence of maps $f_k\in PL_+(I)$.  In particular\\

%\begin{center} $A_{2n+1}\subseteq\displaystyle\bigcup_{k\in\omega}f_kA=\bigcup_{k\in\omega}\bigcup_{1\leq i\leq 2n}f_kA_i$.
%\end{center}
%\vspace{.3cm}

%But this is impossible by Lemma \ref{lemma_category} (2), as the union on the right-hand side is only countable, but each $f_kA_i$ is meager in $A_{2n+1}$ in $(A_{2n+1},\tau_0)$ for each $1\leq i\leq 2n+1$.  This contradiction ensures that $PL_+(I)$ has no such Polish topology $\tau$.
\end{proof}

\begin{cor:cor_main}  There is no Polish group topology on any of the following groups:
\begin{enumerate}
		\item $\PL_+(\mathbb{R})$;
		\item $\PLF_+(\mathbb{R})$; and
		\item $\PL_+(\mathbb{S}^1)$.
\end{enumerate}
\end{cor:cor_main}

\begin{proof}  First, we re-define the circle $\mathbb{S}^1$ to be the quotient space $\mathbb{R}/2\mathbb{Z}$ (as opposed to $\mathbb{R}/\mathbb{Z}$ as in the introduction) and re-define $\PL_+(\mathbb{S}^1)$ to be the subgroup of $\PL_+(\mathbb{R})$ consisting of those maps $f$ satisfying $f(x+2)=f(x)+2$ for all $x\in\mathbb{R}$.  This convenience will allow us to prove all three parts of the corollary simultaneously.  Clearly the old and the new definitions yield isomorphic groups.

So let $G$ be any of the groups $\PL_+(\mathbb{R})$, $\PLF_+(\mathbb{R})$, or $\PL_+(\mathbb{S}^1)$.  Suppose for a contradiction that $G$ carries some Polish group topology.  In cases (1) and (2) let $M=\mathbb{R}$; in case (3) let $M=\mathbb{S}^1=\mathbb{R}/2\mathbb{Z}$.  Note that $\PL_+(I)$ naturally embeds into $G$ via identification with the subgroup\\

\begin{center} $H=\{h\in G:h$ fixes $M\backslash[0,1]$ pointwise$\}$
\end{center}
\vspace{.3cm}

\noindent consisting of elements which are supported only on $[0,1]$.

For an element $g\in G$, let $C(g)$ denote the centralizer of $g$.  Each $C(g)$ is closed in $G$, as centralizers are closed in any Hausdorff topological group.  Set $X=\{g\in\PL_+(\mathbb{R}):g$ fixes $[0,1]$ pointwise$\}$.  It is straightforward to verify that $h\in H$ if and only if $h$ commutes with every element of $X$; in other words $H=\displaystyle\bigcap_{g\in X}C(g)$ and $H$ is closed in $G$.  This implies that $H\cong\PL_+(I)$ inherits a Polish group topology from $G$, contrary to Theorem \ref{thm_main}.
\end{proof}

\section{Related Results}

Recall that $\Homeo^{Lip}_+(I)$ is the subset of $\Homeo_+(I)$ consisting of maps $f$ which are both Lipschitz and have Lipschitz inverse, i.e. there is a constant $C$ such that $|f(x)-f(y)|<C|x-y|$ and $|f^{-1}(x)-f^{-1}(y)|<C|x-y|$ for all $x,y\in I$.  It is clear that $\Homeo^{Lip}_+(I)$ is a subgroup of $\Homeo_+(I)$.

Recall also that for a fixed $0<\epsilon<1$, $\Diff^{1+\epsilon}_+(I)$ denotes the subset of $\Diff_+(I)$ consisting of diffeomorphisms $f$ whose derivatives satisfy the H\"older condition with exponent $\epsilon$: that is, there is a constant $C$ such that $|f'(x)-f'(y)|<C|x-y|^\epsilon$ for all $x,y\in I$.  The fact that $\Diff^{1+\epsilon}_+(I)$ comprises a group (with no additional requirement on the inverse) may be deduced from the chain rule for derivatives.

\begin{thm:thm_also_1}  There is no Polish group topology on $\Homeo^{Lip}_+(I)$.
\end{thm:thm_also_1}

\begin{proof}  Let $\Sigma$ denote the class of Borel sets of $\Homeo^{Lip}_+(I)$ inherited as a subgroup of $\Homeo_+(I)$.  For each $n\in\omega$, let\\

\begin{center} $B_n=\{f\in\Homeo^{Lip}_+(I):|f(x)-f(x)|<n|x-y|$ and $|f^{-1}(x)-f^{-1}(y)|<n|x-y|$ for all $x,y\in I\}$,
\end{center}
\vspace{.3cm}

\noindent so $\Homeo^{Lip}_+(I)=\displaystyle\bigcup_{n\in\omega}B_n$.  Note that for $f\in\Homeo_+(I)$,\\

\begin{center} $f\in B_n~\leftrightarrow~\forall p, q\in (I\cap\mathbb{Q})~(p\neq q)~\rightarrow~ \left|\dfrac{f(p)-f(q)}{p-q}\right|<n~\wedge~\left|\dfrac{f^{-1}(p)-f^{-1}(q)}{p-q}\right|<n$,
\end{center}
\vspace{.3cm}

\noindent from which we see that $B_n$ is $G_\delta$ in $\Homeo_+(I)$, i.e. $B_n\in\Sigma$, and $\Homeo^{Lip}_+(I)$ is a $\mathbf{\Sigma}^0_3$ (or $G_{\delta\sigma}$) subset of $\Homeo_+(I)$.  Note each $B_n$ is symmetric and satisfies $B_n^2\subseteq B_{n^2}$; thus conditions (1), (2), and (3) of Lemma \ref{lemma_standard} are all satisfied.

It remains to check condition (4).  Fix $n\in\omega$ and let $(g_k)_{k\in\omega}$ be an arbitrary sequence of elements in $\Homeo^{Lip}_+(I)$.  We will use a diagonal argument to produce a bi-Lipschitz map $f$ such that $fg_k^{-1}\notin B_n$ for each $k$.  To simplify terminology: if $f\in\Homeo_+(I)$ and $J$ is a subinterval of $I$, say that $f$ is $C$-Lipschitz on $J$ if $|f(x)-f(y)|<C|x-y|$ whenever $x,y\in J$.  So $A_n$ is the class of maps $f$ for which $f$ and $f^{-1}$ are both $n$-Lipschitz on $I$.

Let $(J_k)_{k\in\omega}$ be a sequence of pairwise disjoint nonempty open subintervals of $I$.  Write $J_k=(a_k,b_k)$, and let $\ell_k=b_k-a_k$ denote the length of $J_k$.  For each interval $J_k$, define the partial homeomorphism $f\upharpoonright_{J_k}:J_k\rightarrow J_k$ by cases as follows.\\

\begin{enumerate}
		\item If $g_k$ is not $n$-Lipschitz on $J_k$ or if $g_k^{-1}$ is not $n$-Lipschitz on $g_k(J_k)$: let $f\upharpoonright_{J_k}$ be the identity map on $J_k$.\\
		\item If $g_k$ is $n$-Lipschitz on $J_k$ and $g_k^{-1}$ is $n$-Lipschitz on $g_k(J_k)$: Let $x_k=a_k+\frac{\ell_k}{n^2+1}$.  Let $f\upharpoonright_{J_k}$ be the increasing piecewise linear self-homeomorphism of $J_k$ which has exactly one break point at $x_k$, and maps $x_k$ to $f(x_k)=b_k-\frac{\ell_k}{n^2+1}$.  (See Figure A.)
\end{enumerate}
\vspace{.3cm}

\begin{center}
		\includegraphics{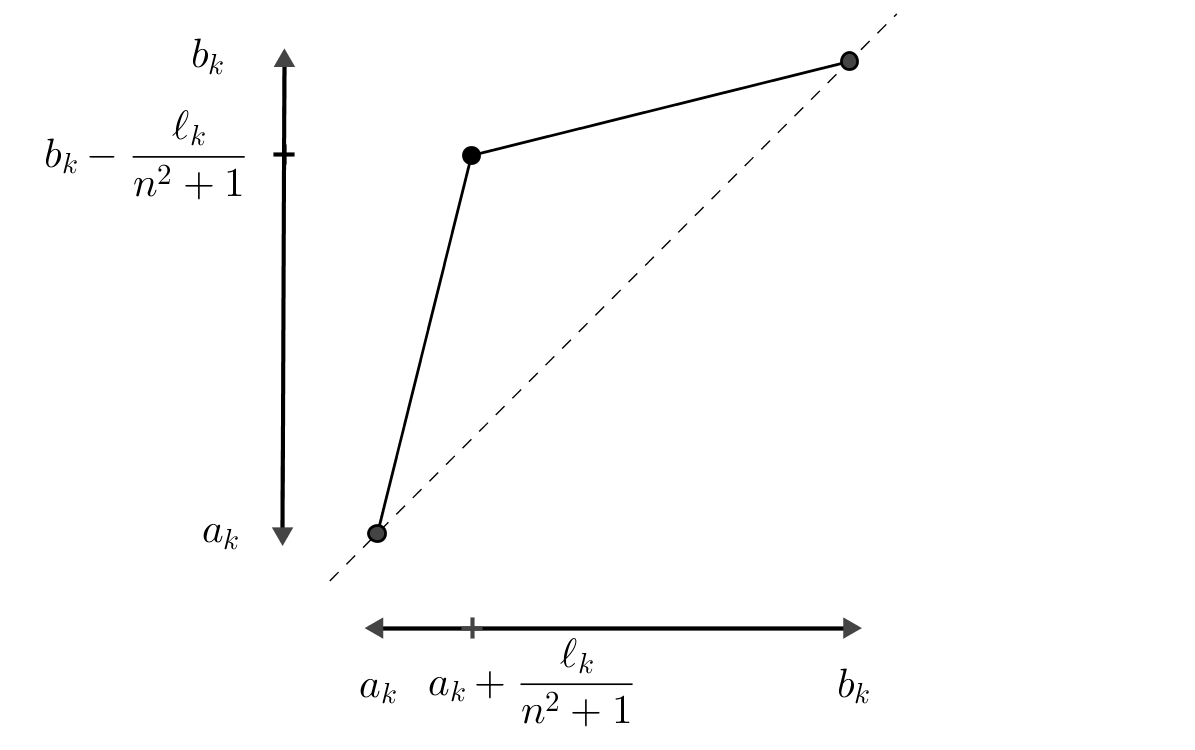}\\
		Figure A: Graph of $f\upharpoonright_{J_k}:J_k\rightarrow J_k$.
\end{center}
\vspace{.3cm}

Define $f$ to be the identity on any points where it is not yet defined; thus we obtain $f\in\Homeo_+(I)$.  We claim $f$ is $(n^2+1)$-Lipschitz on each interval $J_k$.  In the first case of our construction this is immediate.  In the second case, it follows because $f$ is piecewise linear with maximal slope $\frac{f(x_k)-a_k}{x_k-a_k}=\left(\frac{n^2\ell_k}{n^2+1}\right)/\left(\frac{\ell_k}{n^2+1}\right)=n^2<n^2+1$.  The piecewise linear inverse $f^{-1}$ actually has the same maximal slope on $J_k$, and hence is $(n^2+1)$-Lipschitz on $J_k$.  Since this holds for all $J_k$ and $f$ is the identity elsewhere, it is straightforward to deduce from the triangle inequality that $f$ is $(n^2+1)$-Lipschitz on all of $I$, i.e. $f\in B_{n^2+1}\subseteq\Homeo^{Lip}_+(I)$.

However, $fg_k^{-1}$ cannot lie in $B_n$ for any $k\in\omega$.  For if condition (1) in the construction holds, then $fg_k^{-1}$ coincides with $g_k^{-1}$ on $g(I_k)$ and $g_kf^{-1}$ coincides with $g_k$ on $J_k$, whence either $fg_k^{-1}$ or its inverse violates the $n$-Lipschitz condition.

On the other hand if condition (2) holds, then $fg_k^{-1}$ maps the points $g_k(a_k)$, $g_k(x_k)$ to a distance $\frac{n^2\ell_k}{n^2+1}$ away from one another.  Since $x_k-a_k=\frac{\ell_k}{n^2+1}$ and $g_k$ is $n$-Lipschitz on $J_k$, we have $g(x_k)-g(a_k)< n\cdot\frac{\ell_k}{n^2+1}=\frac{n\ell_k}{n^2+1}$.  So\\

\begin{center} $\dfrac{fg_k^{-1}(g_k(x_k))-fg_k^{-1}(g_k(a_k))}{g_k(x_k)-g_k(a_k)}>\left(\dfrac{n^2\ell_k}{n^2+1}\right)/\left(\dfrac{n\ell_k}{n^2+1}\right)=n$
\end{center}
\vspace{.3cm}

\noindent and therefore $fg_k^{-1}$ is not $n$-Lipschitz on $g_k(J_k)$.

\noindent Thus we have produced $f\in\Homeo^{Lip}_+(I)$ such that $fg_k^{-1}\notin B_k$ for each $k$, i.e. $f\notin\displaystyle\bigcup_{k\in\omega}B_kg_k$.  The fact that $(\Homeo^{Lip}_+(I),\Sigma)$ is non-Polishable now follows from Lemma \ref{lemma_standard}, and that it has no Polish topology from Lemma \ref{lemma_main}.
\end{proof}

\begin{thm:thm_also_2}  There is no Polish group topology on $\Diff^{1+\epsilon}_+(I)$.
\end{thm:thm_also_2}

\begin{proof}  We proceed similarly to the previous proof.  Take $\Sigma$ to be the Borel subsets of $\Diff^{1+\epsilon}_+(I)$ inherited as a subgroup of $\Homeo_+(I)$.  For $f\in\Homeo_+(I)$ and for an open subinterval $J$ of $I$, say that $f'$ is $C$-H\"older on $J$ if $|f'(x)-f'(y)|<C|x-y|^\epsilon$ for all $x,y\in J$.  Define\\

\begin{center}  $B_n=\{f\in\Diff^{1+\epsilon}_+(I):f'$ and $(f^{-1})'$ are $n$-H\"older on $I\}$.
\end{center}
\vspace{.3cm}

So $\Diff^{1+\epsilon}_+(I)=\displaystyle\bigcup_{n\in\omega}B_n$.  For $f\in\Diff_+(I)$ we have\\

\begin{center} $f\in B_n~\leftrightarrow~\forall p, q\in(I\cap\mathbb{Q})~(p\neq q)~\rightarrow~ \dfrac{|f'(p)-f'(q)|}{|p-q|^\epsilon}<n~ \wedge \dfrac{|(f^{-1})'(p)-(f^{-1})'(q)|}{|p-q|^\epsilon}<n$
\end{center}
\vspace{.3cm}

\noindent and therefore $B_n$ is $G_\delta$ in $\Diff_+(I)$ (where the latter is topologized as a Polish group with the $C^1$-metric).  $\Diff_+(I)$ injects continuously into $\Homeo_+(I)$, and therefore $B_n$ is Borel in $\Homeo_+(I)$ as well by the Lusin-Souslin theorem, i.e. $B_n\in\Sigma$.  So (1) and (2) in Lemma \ref{lemma_standard} hold.

To show that (3) holds, we must make a few comments.  First, we claim that if $f\in B_n$, then its derivative $f'$ is uniformly bounded above on $I$ by say $n+1$.  For if not, there would exist a point $x_0\in I$ for which $f'(x_0)>n+1$.  But by the H\"older condition, for every other $x\in I$, $|f'(x)-f'(x_0)|<n|x-x_0|^\epsilon\leq n$; whence $f'(x)>f'(x_0)-n>1$.  In other words $f'$ is bounded below strictly by $1$, which is clearly impossible for a self-diffeomorphism of the interval- this proves the claim.  So we have $||f'||_\infty\leq n+1$ for each $f\in B_n$, where $||f'||_\infty$ denotes the least upper bound of $f'$ on $I$.

Now it follows from the chain rule that if $f$ and $g$ are diffeomorphisms for which $f'$ and $g'$ are both $n$-H\"older, then the derivative of the product $(fg)'$ is $(||f'||_\infty+||g'||_\infty^{1+\epsilon})n$-H\"older.  This fact, combined with our remarks in the previous paragraph, implies that any product of elements $f$ and $g$ from $B_n$ will actually be $m_n$-H\"older where $m_n$ is the least integer greater than $[(n+1)+(n+1)^{1+\epsilon}]n$, i.e. $A_n^2\subseteq A_{m_n}$.  This shows that Lemma \ref{lemma_standard} (3) holds for the sequence $(B_n)$.

As in the previous proof, we achieve (4) by a diagonal argument, although considerably more technical care is required.  Let $(J_k)_{k\in\omega}$ be a sequence of disjoint nonempty open subintervals of $I$.  Write $J_k=(a_k,b_k)$ and set $\ell_k=b_k-a_k$.  Furthermore assume that the sequence $(J_k)$ is chosen in such a way that $\displaystyle\inf_{x\in J_k,y\in J_p}|x-y|\geq\max\{\ell_k,\ell_p\}$ for each pair of distinct integers $k$ and $p$.  We define the derivative $f'$ by cases as follows:

\begin{enumerate}
		\item If $g_k'$ is not $n$-H\"older on $J_k$ or $(g_k^{-1})'$ is not $n$-H\"older on $g(J_k)$: let $f'\upharpoonright_{J_k}$ be identically $1$.\\
		\item If $g_k'$ is $n$-H\"older on $J_k$ and $(g_k^{-1})'$ is $n$-H\"older on $g(J_k)$:  Let $x_k$ be the midpoint of $J_k$, and let $y_k=x_k+\dfrac{\ell_k}{2(n+1)^4}$.  Now let $f'$ be the piecewise linear map $J_k$ which maps both endpoints to $1$, whose only break points are $x_k$, $y_k$, and $x_k-(y_k-x_k)$, which maps $x_k$ to $f'(x_k)=1+(\ell_k)^\epsilon(n+1)^4$ and $y_k$ to $f'(y_k)=1-(\ell_k)^\epsilon$, and which is symmetric about the line $x=x_k$.  (See Figure B.)\\
\end{enumerate}
\vspace{.3cm}

\begin{center} \includegraphics{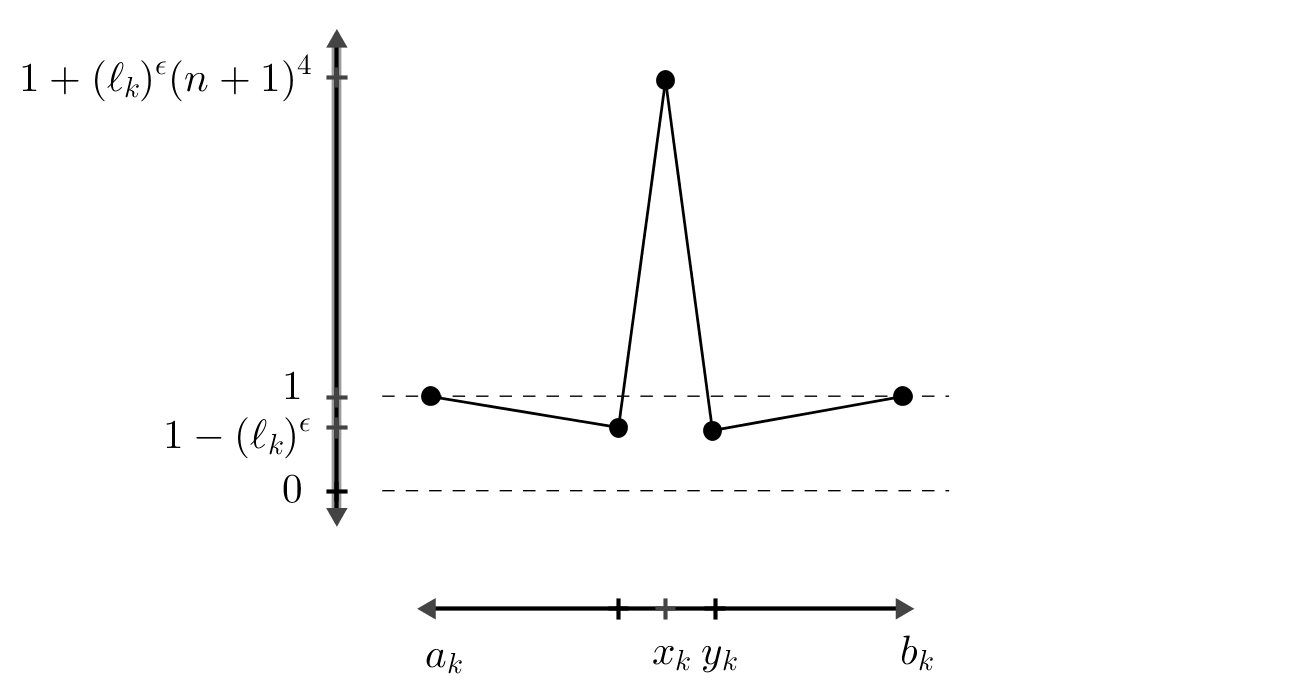}\\
Figure B: Graph of $(f')\upharpoonright_{J_k}:J_k\rightarrow J_k$.
\end{center}
\vspace{.3cm}

Now a straightforward area calculation shows that $\int_{a_k}^{b_k}f'(x)dx$ is equal to $\ell_k$, and thus the antiderivative $f$ of $f'$ which fixes $a_k$ is a self-diffeomorphism of $J_k$.  Letting $f$ be the identity wherever we have not yet defined it, $f'$ becomes continuous on $I$ and thus we have $f\in\Diff_+(I)$.  (To check continuity of the derivative at the endpoints of $I$, observe that $\displaystyle\lim_{k\rightarrow\infty}(\ell_k)^\epsilon=0$.)  On each interval $J_k$, the derivative $f'$ is piecewise linear and by inspection, the slopes of the linear segments are bounded above by $C(\ell_k)^{\epsilon-1}$, where $C=2((n+1)^4+1)(n+1)^4$.  This implies that for any $x,y\in J_k$, we have\\

\begin{center} $\dfrac{|f'(x)-f'(y)|}{|x-y|^\epsilon}=\dfrac{|f'(x)-f'(y)|}{|x-y|}\cdot|x-y|^{1-\epsilon}< C(\ell_k)^{\epsilon-1}\cdot(\ell_k)^{1-\epsilon}=C$.
\end{center}
\vspace{.3cm}

So $f'$ is $C$-H\"older on each $J_k$.  

To see that $f'$ is $C$-H\"older continuous on all of $I$, let $x,y\in I$ be arbitrary.  If neither $x$ nor $y$ lies in some subinterval $J_k$, then $f'(x)=f'(y)=1$ and $f'$ obviously satisfies the H\"older condition at these points.  So assume $x\in J_k$ for some integer $k$.  Assume by permuting $x$ and $y$ if necessary that if $y\in J_p$ for some integer $p$, then $\ell_k\geq\ell_p$.  Now either $f'(x)\geq f'(y)$ or $f'(x)<f'(y)$.  In the first case, set $z=y_k$.  Since $y\in J_p$ and $J_p$ is a smaller interval than $J_k$, we have $f'(y)\geq f'(y_p)\geq f'(y_k)=f'(z)$, and therefore $|f'(x)-f'(z)|=f'(x)-f'(z)\geq f'(x)-f'(y)=|f'(x)-f'(y)|$.  In the second case, set $z=x_k$; then since $y\in J_p$, we have $f'(y)\leq f'(x_p)\leq f'(x_k)=f(z)$, whence $|f'(x)-f'(z)|=f'(z)-f'(x)\geq f'(y)-f'(x)=|f'(x)-f'(y)|$.  In both cases we have $|f'(x)-f'(z)|\geq|f'(x)-f'(y)|$.  But $x,z\in J_k$, and hence $|x-z|\leq\ell_k\leq|x-y|$ since we chose our intervals $J_k$ and $J_p$ to be at least at a distance $\ell_k$ apart from one another.  Thus the fact that $f'$ satisfies the $C$-H\"older condition at $x$ and $y$ now follows from the fact that it does so at $x$ and $z$ in $J_k$.  Thus $f'$ is $C$-H\"older on all of $I$, and our constructed $f$ indeed lies in $\Diff_+^{1+\epsilon}(I)$.

We claim that for each $k\in\omega$, either the derivative of $fg_k^{-1}$ is not $n$-H\"older on the interval $g(J_k)$ or else the derivative of $(fg_k^{-1})^{-1}=g_kf^{-1}$ is not $n$-H\"older on the interval $J_k$.  In case 1 of our construction, we have either $fg_k^{-1}=g_k^{-1}$ on $g(J_k)$ or $g_kf^{-1}=g_k$ on $J_k$, from which the conclusion follows.  So we consider case 2 of the construction.  Since $g_k'$ is $n$-H\"older on $J_k$, it is also bounded above by $n+1$ (see the argument in the third paragraph of this proof), and thus $g_k$ is $(n+1)$-Lipschitz on $J_k$.  It follows that\\

\begin{center} $|g_k(x_k)-g_k(a_k)|\leq(n+1)\ell_k$.
\end{center}
\vspace{.3cm}

On the other hand, we have\\

\begin{align*}
|(fg_k^{-1})'(g_k(x_k))-(fg_k^{-1})'(g_k(a_k))| &= \left|\dfrac{f'(x_k)}{g_k'(x_k)}-\dfrac{f'(a_k)}{g_k'(a_k)}\right|\\
&\geq \dfrac{1}{|g_k'(x_k)|}|f'(x_k)-f'(a_k)|-|f'(a_k)|\left|\dfrac{1}{g_k'(x_k)}-\dfrac{1}{g_k'(a_k)}\right|.
\end{align*}
\vspace{.3cm}

Now looking at the terms in the last line above, note that by our construction we have $|f'(x_k)-f'(a_k)|=(\ell_k)^\epsilon(n+1)^4$.  On the other hand $|g_k'(x_k)|$ is bounded above by $n+1$; $|f'(a_k)|$ is exactly $1$; and by the $n$-H\"older condition applied to $g_k^{-1}$,\\

\begin{align*}
\left|\dfrac{1}{g_k'(x_k)}-\dfrac{1}{g_k'(a_k)}\right| &= |(g_k^{-1})'(g_k(x_k))-(g_k^{-1})'(g_k(a_k))|\\
&\leq n|g_k(x_k)-g_k(a_k)|^\epsilon\\
&\leq n[(n+1)\ell_k]^\epsilon\\
&< (n+1)^2(\ell_k)^\epsilon.
\end{align*}
\vspace{.3cm}

It follows from the crude estimates above that\\

\begin{center} $|(fg_k^{-1})'(g_k(x_k))-(fg_k^{-1})'(g_k(a_k))|\geq\dfrac{(\ell_k)^\epsilon(n+1)^4}{n+1}-(n+1)^2(\ell_k)^\epsilon=n(n+1)^2(\ell_k)^\epsilon$
\end{center}
\vspace{.3cm}

\noindent and therefore\\

\begin{center} $\dfrac{|(fg_k^{-1})'(g_k(x_k))-(fg_k^{-1})'(g_k(x_a))|}{|g_k(x_k)-g_k(a_k)|^\epsilon}\geq\dfrac{n(n+1)^2(\ell_k)^\epsilon}{(n+1)^\epsilon(\ell_k)^\epsilon}>n(n+1)>n$.
\end{center}
\vspace{.3cm}

So $(fg_k^{-1})'$ is not $n$-H\"older and $f\notin B_ng_k$ for any $k\in\omega$.  This shows condition (4) of Lemma \ref{lemma_standard} holds (and thus all conditions of the lemma hold) for the sequence $(B_n)$ in $\Diff_+^{1+\epsilon}(I)$, and we conclude that $(\Diff_+^{1+\epsilon}(I),\Sigma)$ is not Polishable.  So $\Diff_+^{1+\epsilon}(I)$ has no Polish group topology by Lemma \ref{lemma_main}.
\end{proof}

\begin{cor:cor_also}  There is no Polish group topology on either of the following groups:
\begin{enumerate}
		\item $\Homeo_+^{Lip}(\mathbb{R})$; and
		\item $\Homeo_+^{Lip}(\mathbb{S}^1)$.
\end{enumerate}
\end{cor:cor_also}

\begin{proof}  As in the proof of Corollary \ref{cor_main}, for convenience we take $\mathbb{S}^1=\mathbb{R}/2\mathbb{Z}$.  Let $G$ be either group listed above.  Then $\Homeo_+^{Lip}(I)$ may be naturally identified with the subgroup $H$ of $G$ consisting of those elements which fix the complement of $[0,1]$.  Therefore we may proceed exactly as in the proof of Corollary \ref{cor_main} to show that $\Homeo_+^{Lip}(I)$ in fact embeds as a closed subgroup of $G$, whenever $G$ is given a Hausdorff group topology.  Then it follows from Theorem \ref{thm_also_1} that $G$ cannot be made Polish, else $H$ would be Polish as well.
\end{proof}

\begin{rem} \label{rmk_extend}  Of course the group $\Diff_+^{1+\epsilon}(\mathbb{S}^1)$ (defined in the obvious way) holds great interest and we would like to show that it has no compatible Polish group topology.  Unfortunately the direct argument used to obtain both Corollaries \ref{cor_main} and \ref{cor_also} does not appear to apply in this case, as $\Diff_+^{1+\epsilon}(I)$ may not be identified with the subgroup of $\Diff_+^{1+\epsilon}(\mathbb{S}^1)$ consisting of elements which are supported only on a proper sub-arc of $\mathbb{S}^1$.  This is simply because the diffeomorphisms in $\Diff_+^{1+\epsilon}(I)$ may have any positive derivative at the endpoints $0$ and $1$; but diffeomorphisms supported on some sub-arc $[a,b]$ must have derivative $1$ at $a$ and $b$ to ensure continuity of the derivative.  Despite this difficulty, the ingredients to the proof of Theorem \ref{thm_also_2} all hold if $I$ is replaced with $\mathbb{S}^1$; in particular, Lemmas \ref{lemma_kallman}, \ref{lemma_hayes}, and \ref{lemma_main} are all still true (see \cite{kallman_1986a} and \cite{hayes_1997a} respectively for the proofs), and the construction of the map $f$ in the proof of Theorem \ref{thm_also_2} still makes sense and will still yield a H\"older-continuous function which misses each set $g_kB_n$, as long as the subintervals $(J_k)$ of $\mathbb{S}^1$ are chosen to be sufficiently separated.  So although we omit the details here, there appear to be no obstructions in extending the proof of Theorem \ref{thm_also_2} to apply also to the group $\Diff_+^{1+\epsilon}(\mathbb{S}^1)$.
\end{rem}

As a final remark, we observe that our results hold for each of these groups of transformations we have considered even if we allow the orientation-reversing ones to be included.  In what follows, the omission of the subscript $+$ indicates that we mean the full group of homeomorphisms satisfying a given property, regardless of whether it is orientation-preserving or reversing.  So, for instance, $\PL(I)$ is just the group of all piecewise linear homeomorphisms of $I$.

\begin{cor}  There is no Polish group topology on any of the following groups:
\begin{enumerate}
		\item $\PL(I)$;
		\item $\Homeo^{Lip}(I)$;
		\item $\Diff^{1+\epsilon}(I)$;
		\item $\PL(\mathbb{R})$;
		\item $\PLF(\mathbb{R})$;
		\item $\Homeo^{Lip}(\mathbb{R})$;
		\item $\PL(\mathbb{S}^1)$; and
		\item $\Homeo^{Lip}(\mathbb{S}^1)$.
\end{enumerate}
\end{cor}

\begin{proof}  In each case, by the results of \cite{kallman_1986a} and \cite{hayes_1997a}, any Hausdorff group topology on the group in question must contain the compact-open topology.  In that case the subgroup consisting of orientation-preserving maps must be closed; the corollary now follows from the main results of this article.
\end{proof}

\end{document}